\documentclass[preprint,12pt]{elsarticle}

\usepackage{amssymb}
\usepackage{amsthm}

\newcommand{\sysn}{\left\{\begin{array}{rcl}}
\newcommand{\sysk}{\end{array}\right.}

\newtheorem{theorem}{Theorem}[section]

\theoremstyle{example}

\theoremstyle{definition}
\newtheorem{definition}[theorem]{Definition}

\newtheorem{corollary}[theorem]{Corollary}

\journal{...}

\begin{document}

\begin{frontmatter}



\title{On some properties of the space of upper semicontinuous functions}


\author[label1]{Alexander V. Osipov}

\ead[label1]{OAB@list.ru}

\address[label1]{Krasovskii Institute of Mathematics and Mechanics, Ural Federal
 University, Ural State University of Economics, 620219, Ekaterinburg, Russia}

\author[label2]{Evgenii G. Pytkeev}

\ead[label2]{pyt@imm.uran.ru}


\address[label2]{Krasovskii Institute of Mathematics and Mechanics, Ural Federal
 University, 620219, Ekaterinburg, Russia}

\begin{abstract}
For a Tychonoff space $X$, we will denote by $USC_{p}(X)$
($B_1(X)$) the set of all real-valued upper semicontinuous
functions (the set of all Baire functions of class 1) defined on
$X$ endowed with the pointwise convergence topology.

In this paper we describe a class of Tychonoff spaces $X$ for
which the space $USC_{p}(X)$ is sequentially separable.
Unexpectedly, it turns out that this class coincides with the
class of spaces for which a stronger form of the sequential
separability for the space $B_1(X)$ holds.


\end{abstract}


\end{frontmatter}



\section{Introduction}

If $X$ is a topological space and $A\subseteq X$, then the
sequential closure of $A$,
 denoted by $[A]_{seq}$, is the set of all limits of sequences
 from $A$. A set $D\subseteq X$ is said to be sequentially dense
 if $X=[D]_{seq}$. If $D$ is a countable sequentially dense subset
 of $X$ then $X$ is called sequentially separable space \cite{20,23}.

\label{} Let $X$ be a Tychonoff space. We consider the following
function spaces.

$\bullet$  $C_{p}(X)$ is the set of all real-valued continuous
functions defined on $X$ endowed with the topology of pointwise
convergence.

$\bullet$ $B_1(X)$ is the set of all Baire functions class 1
(i.e., pointwise limits of continuous functions) defined
 on $X$ endowed with the pointwise
convergence topology.

$\bullet$ $USC_p(X)$ is the set $USC(X)=\{f\in \mathbb{R}^X :
f^{-1}((-\infty, r))$ is an open set of $X$ for any $r\in
\mathbb{R} \}$ (i.e. the set all upper semicontinuous functions
defined
 on $X$) endowed with the topology of pointwise
convergence.

\medskip

Note that $C_p(X)\subseteq USC_p(X)\subseteq B_1(X)$ for a
separable metrizable space $X$. It follows that $USC_p(X)$ is
sequentially separable for a separable metrizable space $X$
(Theorem \ref{th22}).

It is well known that $f: X \rightarrow \mathbb{R}$ is a Baire
function if and only if there exists a continuous mapping
$\varphi: X \rightarrow M$ from $X$ onto a separable metrizable
space $M$ and a Borel function $g: M \rightarrow \mathbb{R}$ such
that $f=g\circ \varphi$. If we replace the Borel function by upper
semicontinuous function in this characterization, we obtain the
function $f: X\rightarrow \mathbb{R}$ such that $f^{-1}((-\infty,
r))= \varphi^{-1}(g^{-1}((-\infty, r)))$ is a cozero-set of $X$
for any $r\in \mathbb{R}$.

Define the function space $USC^f(X):=\{f\in \mathbb{R}^X :
f^{-1}((-\infty, r))$ is a cozero set of $X$ for any $r\in
\mathbb{R} \}$. Clearly, that if $X$ is a perfectly normal space
then $USC^f(X)=USC(X)$. We denote by $USC^f_p(X)$ the set
$USC^f(X)$ endowed with the topology of pointwise convergence.

We claim that $USC_p^f(X)$ is sequentially separable if and only
if there exists a countable subset $S$ of $C(X)$ such that
$[S]_{seq}=B_1(X)$ (i.e., when a stronger form of the sequential
separability for the space $B_1(X)$ holds).

\section{Main definitions and notation}

 We recall that a subset of $X$ that is the
 complete preimage of zero for a certain function from~$C(X)$ is called a zero-set.
A subset $O\subseteq X$  is called  a cozero-set of $X$ if
$X\setminus O$ is a zero-set. If a set $Z=\bigcup_{i\in
\mathbb{N}} Z_i$ where $Z_i$ is a zero-set of $X$ for any $i\in
\mathbb{N}$, then $Z$ is called $Z_{\sigma}$-set of $X$. Note that
if a space $X$ is a perfectly normal space, then class of
$Z_{\sigma}$-sets of $X$ coincides with class of $F_{\sigma}$-sets
of $X$. It is well known \cite{rj}, that $f\in B_1(X)$ if and only
if $f^{-1}(G)$ - $Z_{\sigma}$-set for any open set $G$ of real
line $\mathbb{R}$.


\medskip
Recall that the $i$-weight $iw(X)$ of a space $X$ is the smallest
infinite cardinal number $\tau$ such that $X$ can be mapped by a
one-to-one continuous mapping onto a Tychonoff space of the weight
not greater than $\tau$.

\begin{theorem} (Noble's Theorem in \cite{nob}) \label{th31} Let $X$ be a Tychonoff space. A space $C_{p}(X)$ is separable if and only if
$iw(X)=\aleph_0$.
\end{theorem}

\begin{theorem} \label{th30} (Pestriakov's Theorem in \cite{ps}). Let $X$ be a Tychonoff space. A space $B_{1}(X)$ is separable if and only if $iw(X)=\aleph_0$.
\end{theorem}

\begin{definition} A Tychonoff space $X$ has  the Velichko property ($X$ $\models$ $V$), if there
 exists  a condensation (one-to-one continuous mapping) $f: X \mapsto Y$ from the space $X$ on a
 separable metric space $Y$, such that $f(U)$ is an $F_{\sigma}$-set
 of $Y$ for any cozero-set $U$ of $X$.
\end{definition}

\begin{theorem} \label{th38} (Velichko \cite{vel}). Let $X$ be a Tychonoff space. A space $C_p(X)$ is
sequentially separable if and only if  $X$ $\models$ $V$.
\end{theorem}
\medskip

\begin{theorem} (\cite{vel}) \label{th32} A space $B_1(X)$ is
sequentially separable for any separable metric space $X$.
\end{theorem}
\medskip

Note that proof of this theorem  gives more, namely that there
exists a countable subset $S\subset C_p(X)$, such that
$[S]_{seq}=B_1(X)$.

Hence, a space $USC_p(X)$  is sequentially separable for any
separable metric space $X$.

\medskip

In \cite{ospy}, Osipov and Pytkeev have established criterion for
$B_{1}(X)$ to be sequentially separable.

\begin{definition} A space $X$ has {\bf $OP$-property} ($X$ $\models$ $OP$), if there
 exists a bijection $\varphi: X \mapsto Y$ from a space $X$ onto a
 separable metrizable space $Y$, such that

\begin{enumerate}

\item $\varphi^{-1}(U)$ is a $Z_{\sigma}$-set of $X$ for any open
set $U$ of $Y$;

\item  $\varphi(T)$ is an $F_{\sigma}$-set of $Y$ for any zero-set
$T$ of $X$.

\end{enumerate}
\end{definition}

\begin{theorem}(Theorem 3.1 in \cite{ospy})\label{th6}
 A function space $B_1(X)$ is sequentially separable if and only if  $X$ $\models$ $OP$.
\end{theorem}

\begin{theorem}(Example 3.3  in \cite{ospy})\label{th61}
There is a Tychonoff space $X$ such that $C_p(X)$ is sequentially
separable, but $B_1(X)$ is not.

\end{theorem}

In the above theorem, the promised space $X$ could be, for
example, if $X$ is the Sorgenfrey line (or the Niemytzki plane)
\cite{ospy}.

\section{Main results}

\begin{definition} A space $X$ has {\bf $U$-property} ($X$ $\models$
$U$), if there
 exists  a condensation $f: X \mapsto Y$ from the space $X$ onto a
 separable metric space $Y$, such that $f(D)$ is a $Z_{\sigma}$-set
 of $Y$ for any zero-set $D$ of $X$.
\end{definition}

\begin{theorem}\label{th22} Let $X$ be a Tychonoff space. Then the following statements are
equivalent:

\begin{enumerate}

\item $USC^f_p(X)$ is sequentially separable;

\item $X$ $\models$ $U$;

\item there exists a countable subset $S$ of $C(X)$ such that
$[S]_{seq}=B_1(X)$.

\end{enumerate}

\end{theorem}

\medskip

\begin{proof}
$(1)\Rightarrow(2)$. Assume that $USC^f_p(X)$ is sequentially
separable. Let $A=\{f_i : i\in \mathbb{N}\}$ be a sequentially
dense subset of $USC^f_p(X)$. Note that $f^{-1}(W)$ is a
$Z_{\sigma}$-set of $X$ for an open set $W$ of $\mathbb{R}$ and
$f\in USC^f_p(X)$. It follows that $USC^f_p(X)$ is a dense subset
of $B_1(X)$ and hence $B_1(X)$ is separable. By Theorem
\ref{th30}, $iw(X)=\aleph_0$. Hence there exists  a condensation
from the space $X$ on a separable metric space $M$. Let $\beta$ be
a countable base of the space $M$. Let $\alpha=\{f_i^{-1}(-\infty,
r) : r\in \mathbb{Q}$ and $i\in \mathbb{N} \}\bigcup \beta$ and
let $\tau$ be a topology on $X$ generating $\alpha$. Denote
$Y=(X,\tau)$. Note that there exists a condensation $f: X \mapsto
Y$ from the space $X$ onto a separable metric space $Y$. By
definition of $\alpha$, $f_i\in USC(Y)$ for each $i\in
\mathbb{N}$.

We will prove that $f(D)$ is an $F_{\sigma}$-set of $Y$ for any
$Z_{\sigma}$-set $D$ of $X$. Fix a $Z_{\sigma}$-set
$D=\bigcup\limits_{i\in \mathbb{N}} D_i$  of $X$ where $D_i$ is a
zero-set of $X$ and $D_i\subset D_{i+1}$ for each $i\in
\mathbb{N}$. Define the function $h$:  $h(D_1)=1$,
$h(D_{i+1}\setminus D_i)=\frac{1}{i+1}$ for each $i\in \mathbb{N}$
and $h(X\setminus D)=0$. By construction of $h$, $D=h^{-1}((0,
+\infty))$.

Note that $h\in USC^f(X)$ and hence there are $\{f_{i_k} : k\in
\mathbb{N}\}\subset A$ such that $f_{i_k} \rightarrow h$
($k\rightarrow \infty$). It follows that  $D=h^{-1}((0,
+\infty))=\bigcup\limits_{j\in \mathbb{N}} \bigcap\limits_{i_k>j}
f^{-1}_{i_k}([\frac{1}{j}, +\infty))$ and hence  $f(D)$ is an
$F_{\sigma}$-set of $Y$.

$(2)\Rightarrow(1)$. Assume that $X$ $\models$ $U$, i.e. there
 is  a condensation (one-to-one continuous mapping) $f: X \mapsto Y$ from the space $X$ on a
 separable metric space $Y$, such that $f(D)$ is an $F_{\sigma}$-set
 of $Y$ for any $Z_{\sigma}$-set $D$ of $X$. Then $USC^f_p(X)\subset
 USC_p(Y)\subset B_1(Y)$. By Velichko's Theorem \ref{th32}, there is
 $A=\{f_i: i\in \mathbb{N}\}\subset C_p(Y)$ such that
 $[A]_{seq}=B_1(Y)$. Note that $C_p(Y)\subset C_p(X)\subset
 USC^f_p(X)\subset B_1(Y)$. It follows that $A$ is a countable
 sequentially dense subset of $USC^f_p(X)$.

$(3)\Rightarrow(2)$. Suppose that exists a countable subset $S$ of
$C(X)$ such that $[S]_{seq}=B_1(X)$. Consider a topology $\tau$
generated by the family  $\alpha=\{f^{-1}(G) : G$ is an open
subset of $\mathbb{R}$ and $f\in S \}$. Denote $Y=(X,\tau)$. Let
$\varphi$ be a identity map from $X$ onto $Y$. By Theorem 3.1 in
\cite{ospy}, $\varphi$ is a bijection such that $\varphi(D)$ is a
$Z_{\sigma}$-set of $Y$ for any zero-set $D$ of $X$. Since
$S\subset C(X)$, $\varphi$ is a condensation.

$(2)\Rightarrow(3)$. Let $X$ $\models$ $U$. By Theorem \ref{th32},
there exists a countable dense subset $L$ of $C_p(Y)$ such that
$[L]_{seq}=B_1(Y)$. Then $S=\{f\circ \varphi : f\in L \}$ is a
countable subset of $C(X)$. Let $\varphi^*(h):=h\circ \varphi$ for
$h\in B_1(Y)$. Then  $\varphi^*: B_1(Y)\mapsto B_1(X)$ is a
first-level Baire isomorphism. It follows that $[S]_{seq}=B_1(X)$.

\end{proof}

\medskip

\begin{corollary} Let $X$ be a Tychonoff space and let $USC^f_p(X)$ be
sequentially separable. Then $C_p(X)$ and $B_1(X)$ are
sequentially separable.
\end{corollary}

\begin{proof} By Theorem \ref{th38}, $C_p(X)$ is sequentially separable.
By Theorem \ref{th6}, $B_1(X)$ is sequentially separable.
\end{proof}

\medskip

\begin{corollary} Let $X$ be a perfectly normal space. Then the following statements are
equivalent:

\begin{enumerate}

\item $USC_p(X)$ is sequentially separable;

\item $X$ $\models$ $U$;

\item there exists a countable subset $S$ of $C(X)$ such that
$[S]_{seq}=B_1(X)$.

\end{enumerate}

\end{corollary}

A continuous image of sequentially separable space is sequentially
separable. Hence {\it cosmic} spaces - the continuous images of
separable metric spaces (space with a countable network) - are
sequentially separable. So, for any separable metric space $X$ (or
more generally, cosmic $X$), $C_p(X)$ is cosmic, and hence
sequentially separable.

\begin{corollary} Let $X$ be a separable metrizable space. Then:

\begin{enumerate}

\item $USC_p(X)$ is sequentially separable;

\item $B_1(X)$ is sequentially separable;

\item there exists a countable subset $S$ of $C(X)$ such that
$[S]_{seq}=B_1(X)$.

\end{enumerate}

\end{corollary}

Recall that an analytic space is a metrizable space that is a
continuous image of a Polish space.

A map $f: X\rightarrow Y$ be called $Z_{\sigma}$-map, if
$f^{-1}(Z)$ is a $Z_{\sigma}$-set of $X$ for any zero-set $Z$ of
$Y$.

\medskip

We need the following theorem as a special case of the Theorem 1
in \cite{kms}.

\begin{theorem}(\cite{kms})\label{th3} Suppose that $\varphi: L \mapsto S$ be a
$Z_{\sigma}$-mapping from an analytic space $L$ onto a cosmic
space $S$. Then $\varphi$ is piecewise continuous.
\end{theorem}

\begin{theorem} There exists a Tychonoff space $X$ such that
$C_p(X)$ and $B_1(X)$ are sequentially separable, but $USC^f_p(X)$
is not.
\end{theorem}

\begin{proof}

Let $X=Z^{\aleph_0}$ where $Z=\mathbb{N}\cup \{p\}$ for $p\in
\mathbb{N}^*=\beta \mathbb{N}\setminus \mathbb{N}$.

Assume that $USC^f_p(X)$ is sequentially separable. Then, by
Theorem \ref{th22}, there exists a condensation $f: X \mapsto Y$
from the space $X$ on a
 separable metrizable space $Y$, such that $f(D)$ is an $F_{\sigma}$-set
 of $Y$ for any $Z_{\sigma}$-set $D$ of $X$. Since $Z$ is a
 continuous image of $\mathbb{N}$, $X$ is a continuous image
 of irrational numbers $\mathbb{P}$, i.e. there is a continuous mapping $\alpha: \mathbb{P}\mapsto
 X$ from $\mathbb{P}$ onto the space $X$. It follows that $f\circ \alpha :
 \mathbb{P} \mapsto Y$ is a continuous mapping, and  hence  $Y$ is an
 analytic space.  By Theorem \ref{th3},  $f^{-1}:Y\mapsto X $ is a piecewise continuous function (i.e. $Y$
admits a closed and disjoint cover $\mathcal{F}=\{F_n :n\in
\mathbb{N}\}$, such that for each $F_n\in \mathcal{F}$ the
restriction $f^{-1}|F_n$ is continuous function). It follows that
$f^{-1}|F_n : F_n \mapsto f^{-1}(F_n)$ is a homeomorphism, and
hence $X=\bigcup\limits_{F_n\in \mathcal{F}} f^{-1}(F_n)$ where
$f^{-1}(F_n)$ is a separable metrizable space for each $n\in
\mathbb{N}$. Since non-empty open set of $X$ is not metrizable,
$f^{-1}(F_n)$ is a closed nowhere dense subset of $X$ for each
$n\in \mathbb{N}$. But $X$ is a Baire space, a contradiction.

\end{proof}

\section{Open questions}

{\bf Question 1}. Suppose that $B_1(X)$ is sequentially separable.
Is then $C_p(X)$ sequentially separable ?

\bigskip

{\bf Question 2}. Suppose that $B_1(X)$ is sequentially separable.
Is then exist first-level Baire isomorphism $F: X\rightarrow M$
between $X$ and a separable metrizable space $M$ ?
\bigskip







{\bf Question 3}. Suppose that $f: X\mapsto Y$ is a first-level
Baire isomorphism between $X$ and a separable metrizable space
$Y$. Is then $C_p(X)$ sequentially separable ?

\medskip

{\bf Acknowledgment.} The authors are grateful to Sergey V.
Medvedev and the anonymous referee for making several suggestions
which improved this paper.





\bibliographystyle{model1a-num-names}
\bibliography{<your-bib-database>}







\end{document}